\documentclass[a4paper,10pt,reqno]{amsart}
\usepackage{amsfonts,amsmath,amssymb,amsthm}
\usepackage{graphicx}
\usepackage{fixmath}
\usepackage{hyperref}
\usepackage{tikz}
\usepackage{mathrsfs}
\usepackage{mathtools}
\usepackage{stmaryrd}
\usepackage{psfrag}
\usepackage{epsf}
\usepackage{url,rotating}
\usepackage[cp1250]{inputenc}
\usepackage{amsmath,amsfonts,amssymb,latexsym}
\usepackage{setspace}
\usepackage{enumitem}
\usepackage{multirow}

\newtheorem{theo}{Theorem}[section]
\newtheorem{prop}[theo]{Proposition}
\newtheorem{lemma}[theo]{Lemma}

\begin{document}

\pagestyle{plain}

\title{Some extremal ratios of the distance and subtree problems in binary trees}

\author{Shuchao Li}
\address{Shuchao Li\\
School of Mathematics and Statistics \\
Central China Normal University \\
Wuhan 430079, P.R. China
}
\email{lscmath@mail.ccnu.edu.cn}

\author{Hua Wang}
\address{Hua Wang\\
Department of Mathematical Sciences \\
Georgia Southern University \\
Statesboro, GA 30460, USA
}
\email{hwang@georgiasouthern.edu}
\thanks{S. Wang is the corresponding author. \\
\hspace*{4.3mm}S. Li was partially supported by the National Natural Science Foundation of China (Grant Nos. 11671164, 11271149), the Program for New Century Excellent Talents in University (Grant No. NCET-13-0817) and H. Wang was partially supported by the Simons Foundation (\#245307).}

\author{Shujing Wang}
\address{Shujing Wang\\
School of Mathematics and Statistics \\
Central China Normal University \\
Wuhan 430079, P.R. China
}
\email{wangsj@mail.ccnu.edu.cn}


\begin{abstract}
Among many topological indices of trees the sum of distances $\sigma(T)$ and the number of subtrees $F(T)$ have been a long standing pair of graph invariants that are well known for their negative correlation. That is, among various given classes of trees, the extremal structures maximizing one usually minimize the other, and vice versa. By introducing the ``local'' versions of these invariants, $\sigma_T(v)$ for the sum of distance from $v$ to all other vertices and $F_T(v)$ for the number of subtrees containing $v$, extremal problems can be raised and studied for vertices within a tree. This leads to the concept of ``middle parts'' of a tree with respect to different indices. A challenging problem is to find extremal values of the ratios between graph indices and corresponding local functions at middle parts or leaves. This problem also provides new opportunities to further verify the the correlation between different indices such as $\sigma(T)$ and $F(T)$. Such extremal ratios, along with the extremal structures, were studied and compared for the distance and subtree problems for general trees (Barefoot, Entringer and Sz\'ekely, Discrete Applied Math 80 (1997); Sz\'ekely and Wang, Electronic Journal of Combinatorics 20 (2013); Sz\'ekely and Wang, Discrete Mathematics 322 (2014)). In this paper this study is extended to binary trees, a class of trees with numerous practical applications in which the extremal ratio problems appear to be even more complicated. After justifying some basic properties on the distance and subtree problems in trees and binary trees, characterizations are provided for the extremal structures achieving two extremal ratios in binary trees of given order. The generalization of this work to $k$-ary trees is also briefly discussed. The findings are compared with the previous established extremal structures in general trees. Lastly some potential future work is mentioned.
\end{abstract}

\maketitle

\section{Introduction}

The study of questions related to distances in graphs dates back to as early as \cite{jordan}, if not earlier, and has applications in many different fields. The sum of distances between vertices in a graph $G$
$$ \sigma(G) = \sum_{u,v \in V(G)} d_G(u,v), $$
where $d_G(u,v)$ is the distance between $u$ and $v$ in $G$, is also well known as the Wiener index \cite{wiener} for its application in chemical graph theory. Numerous research articles have been published on problems related to the Wiener index. Of particular interest to our work are a number of extremal results on the Wiener index in trees \cite{cela, gutman, ejs, rauten, lovasz, nina, degreeseq, zhang2008, zhang2010}.

While the Wiener index is a representative of distance-based graph invariants, one of the first counting-based graph invariants is the number of subtrees, denoted by $F(T)$ for a tree $T$. This concept, in addition to its application in phylogenetic tree reconstruction \cite{dan}, received much attention from mathematicians and computer scientists in recent years. The extremal results on the number of subtrees of a tree have been established for various classes of trees \cite{kirk, shuchao, subtrees, congnum, largest, gray, joc}.

The above mentioned work led to an interesting observation, that among certain class of graphs/trees, the extremal structure that maximizes the Wiener index usually minimizes the number of subtrees, and vice versa. Such a correlation was further analyzed in \cite{correl}.

If $\sigma(T)$ and $F(T)$ are to be considered as the ``global'' functions defined on trees, the distance function at $v$
$$ \sigma_T(v) = \sum_{u\in V(T)} d_T(u,v) $$
and the number of subtrees containing $v$ in $T$ (denoted by $F_T(v)$) are the natural ``local'' versions. Extremal problems on such local functions lead to the definition of ``middle parts'' of a tree, which are collections of vertices that maximize or minimize certain functions. The first such result is on the set of vertices that minimize the distance function, called the {\it centroid} of a tree and denoted by $C(T)$ \cite{jordan, zelinka}. It was shown that $C(T)$ contains one or two adjacent vertices. Another ``middle part'' of a tree, defined as the set of vertices that maximize $F_T(v)$, is called the {\it subtree core} of $T$ and denoted by $Core(T)$ \cite{subtrees}. As further evidence of the correlation between $\sigma(T)$ and $F(T)$, the subtree core was also shown to contain one or two adjacent vertices. Furthermore, it is known that $\sigma_T(v)$ is maximized and $F_T(v)$ is minimized at a leaf vertex.

For vertices $v \in C(T)$ and $u,w \in L(T)$ (leaf set of $T$), the extremal values of  $\sigma_T(w)/\sigma_T(u)$, $\sigma_T(w)/\sigma_T(v)$, $\sigma(T)/\sigma_T(v)$, and $\sigma(T)/\sigma_T(w)$ were determined for a tree of given order in \cite{barefoot}. As an effort to further verify the negative correlation between the distance problem and the subtree problem, the extremal values of $F_T(w)/F_T(u)$, $F_T(w)/F_T(v)$, $F(T)/F_T(v)$,
and $F(T)/F_T(w)$ were determined in \cite{ratio1, ratio2} for trees of given order, where $v \in Core(T)$, and $u,w \in L(T)$---the complete analogue of  \cite{barefoot}. In \cite{ratio1} it was said that ``extremal behavior of {\em fractions}
is always more delicate than that of the numerator and denominator, therefore it is a natural step to see how far duality between Wiener index and the number of subtrees extend when we study extreme values of the ratios''.  See the table in \cite{ratio1} for a nice summary and comparison of the results in \cite{barefoot, ratio1, ratio2}.

{\it Binary trees} are trees in which every internal vertex is of degree 3, note that this is not to be confused with a {\it rooted binary tree} where every internal vertex has degree 3 except the root which has degree 2. The binary tree is an important data storage/search structure in information science, as well as a default model in many applications such as phylogenetic reconstruction. For earlier work on the Wiener index and the number of subtrees in binary trees one may see \cite{rauten, sorder, subtrees, largest} and the references therein.

We will, in this paper, consider some of the extremal ratios for distance and subtree problems in binary trees, further exploring the correlation between these two concepts. First in Section~\ref{sec:pre} we present some basic properties related to the distance function and number of subtrees in binary trees. We then consider, among binary trees of order $n$ (for an even $n$),  the minimum $\sigma_T(w)/\sigma_T(v)$ (for $w \in L(T)$ and $v\in C(T)$) in Section~\ref{sec:min} and minimum $F_T(v)/F_T(w)$ (for $v \in Core(T)$ and $w \in L(T)$) in Section~\ref{sec:minf}. In Section~\ref{sec:con}, we summarize our findings and further comment on the correlation between the distance and subtree problems. Generalizations to $k$-ary trees and topics for future work are also mentioned.

\section{Preliminaries}
\label{sec:pre}

We first present some facts related to distance and subtree problems in trees and binary trees.  Although some of these observations may have been used or established (informally) in other studies, we include their justifications for completeness. These proofs also help us understand the basic techniques in dealing with distance and subtree problems in trees and binary trees.

The first such fact presents a simple but useful condition on a centroid vertex and its neighbor.

\begin{prop}\label{prop:dis}
For a vertex $v \in C(T)$ and its neighbor $u$, we have
$$ n_{uv}(v) \geq n_{uv}(u) $$
with equality if and only if $u \in C(T)$. Here $n_{uv}(v)$ (resp. $n_{uv}(u)$) is the number of vertices in the component containing $v$ (resp. $u$) in $T-uv$.
\end{prop}
\begin{proof}
Let $T_u$ and $T_v$ be the components of $T-uv$ that contain $u$ and $v$, respectively. Then
\begin{eqnarray*}
  \sigma_T(v)&=&\sum_{w\in V(T)}d_T(w,v)\\
   &=& \sum_{w\in V(T_u)}d_T(w,v)+\sum_{w\in V(T_v)}d_T(w,v) \\
   &=&\sum_{w\in V(T_u)}(d_T(w,u)+1)+\sum_{w\in V(T_v)}(d_T(w,u)-1)  \\
   &=& \sum_{w\in V(T)}d_T(w,u)+|V(T_u)|-|V(T_v)| \\
   &=& \sigma_T(u)+n_{uv}(u)-n_{uv}(v).
\end{eqnarray*}
As $v \in C(T)$, we have that $\sigma_T(v)\leq \sigma_T(u)$ with equality if and only if $u\in C(T)$. The conclusion then follows.
\end{proof}

A parallel statement for the subtree problem is the following.

\begin{prop}\label{prop:sub}
For a vertex $v \in Core(T)$ and its neighbor $u$, we have
$$ F_{T_v}(v) \geq F_{T_u}(u) $$
with equality if and only if $u \in Core(T)$. Here $T_v$ ($T_u$) is the component containing $v$ ($u$) in $T-uv$.
\end{prop}
\begin{proof}
Let $F_T(u, v)$ denote the number of subtrees of $T$ that contain both $u$ and $v$. Then
\[
F_T(v) = F_{T_v}(v)+F_T(u,v),  \  \  \ F_T(u) = F_{T_u}(u)+F_T(u,v).
\]

As $v \in Core(T)$,  we have that $F_T(v)\geq F_T(u)$ with equality if and only if $u\in Core(T)$. Hence $$ F_{T_v}(v) \geq F_{T_u}(u)$$
with equality if and only if $u \in Core(T)$.
\end{proof}

Next we consider a rooted version of the extremal distance problem. This observation will be used frequently in our argument.

\begin{prop}\label{prop:rdis}
For a rooted binary tree $T$ with root $r$ and given number of vertices, $\sigma_{T}(r)$ is maximized by the ``rooted binary caterpillar'' (Figure~\ref{fig:bin}) whose removal of a leaf neighbor of $r$ results in a binary caterpillar.
\end{prop}

\begin{figure}[htbp]
\centering
    \begin{tikzpicture}[scale=.6]
        \node[fill=black,circle,inner sep=1.4pt] (t1) at (0,0) {};
        \node[fill=black,circle,inner sep=1.4pt] (t2) at (1,1) {};
        \node[fill=black,circle,inner sep=1.4pt] (t3) at (2,0) {};
        \node[fill=black,circle,inner sep=1.4pt] (t4) at (2,2) {};
        \node[fill=black,circle,inner sep=1.4pt] (t5) at (3,1) {};
				        \node[fill=black,circle,inner sep=1.4pt] (t6) at (3,3) {};
        \node[fill=black,circle,inner sep=1.4pt] (t7) at (4,2) {};
        \node[fill=black,circle,inner sep=1.4pt] (t8) at (4,4) {};
        \node[fill=black,circle,inner sep=1.4pt] (t9) at (5,3) {};

        \draw (t1)--(1.3,1.3);
				\draw (1.7,1.7)--(t8);
        \draw (t2)--(t3);
        \draw (t4)--(t5);
        \draw (t7)--(t6);
        \draw (t8)--(t9);
        \node at (1.5,1.5) {$\ldots$};
\node at (3.8,4.2) {$r$};

        \end{tikzpicture}
\caption{A rooted binary caterpillar}\label{fig:bin}
\end{figure}
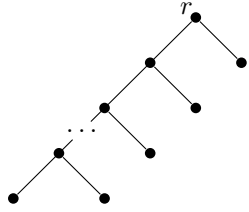

\begin{proof}
Assume that $\sigma_{T}(r)$ is maximized by $T$ and $r_1, r_2$ are two children of $r$. For any $u\in V(T)$, let $T_u$ be the subtree induced by $u$ and the descendants of $u$. We only discuss the non-trivial case where $V(T_{r_1})\ge V(T_{r_2})$ and $V(T_{r_1})\ge 3$. Let $P=rr_1t_2\ldots t_{s-1}x$ be the unique path that connects $r$ and $x$, where $x$ is a vertex at maximum distance from $r$ in $T_{r_1}$. It suffices to prove that $r_2, r_1', t_2', \ldots t_{s-1}'$ are pendent vertices, where $r_1'$ (resp. $t_i'$) is the unique neighbor of $r_1$ (resp. $t_i$) that is not in $V(P)$.

First note that by the extremality of $x$, $t_{s-1}'$ must be a leaf. For any $w\in \{r_2, r_1', t_2',\ldots, t_{s-2}'\}$ that is not a leaf, let $w'$ and $w''$ be the two children of $w$. Consider $T'=T-ww'-ww''+xw'+xw''$, we have
\begin{eqnarray*}
  \sigma_{T'}(r)-\sigma_T(r) &=& \sum_{u\in T_{w}\setminus\{w\}}(d_{T'}(u,r)-d_T(u,r)) \\
   &=& (|T_{w}|-1)(s-d_T(w,r))>0 ,
\end{eqnarray*}
a contradiction.
\end{proof}

Again a parallel statement for the number of subtrees is the following.

\begin{prop}\label{prop:rsub}
For a rooted binary tree $T$ with root $r$ and given number of vertices, $F_{T}(r)$ is minimized by the ``rooted binary caterpillar''.
\end{prop}

\begin{proof}
Assume that $F_{T}(r)$ is minimized by $T$ and $r_1, r_2$ are the two children of $r$. Similar to the previous proof we assume that $V(T_{r_1})\ge V(T_{r_2})$, $V(T_{r_1})\ge 3$ and let $P=rr_1t_2\ldots t_{s-1}x$ be the unique path that connects $r$ and a furthest vertex $x$. Again it suffices to show that the neighbors $r_2, r_1', t_2', \ldots t_{s-1}'$ of the vertices on $P$ are leaves.
Like before, $t_{s-1}'$ must be a leaf.

For any $1\le j\le s-2$, supposing (for contradiction) that $a$ and $b$ are the two children of $t_j'$, let $T'=T-t_j'a-t_j'b+xa+xb$. We now examine the subtrees containing $r$ in $T$ and $T'$. First it is easy to see that the subtrees not containing $x$ or $t'_j$ stay the same in $T$ and $T'$.

Let $S = T-(T_{t'_j}-t'_j)$, from $T$ to $T'$ in order to compare the number of subtrees containing $r$ we only need to compare the number of subtrees containing $r$, $t'_j$ in $T$ and $r$, $x$ in $T'$. It is then easy to see
$$ F_T(r)-F_{T'}(r)=(F_{T_w}(w)-1)(|A|-|B|), $$
where $A$ is the set of subtrees of $S$ that contain $r,t'_j$ but not $x$, and $B$ is the set of subtrees of $S$ that contain $r,x$ but not $t'_j$.

To compare $|A|$ and $|B|$ we establish the following map:
$$ f:B\rightarrow A: \hbox{ {\it For any $R \in B$, let $f(R)=R-t_{s-1}x+t_jt_j'$. } } $$

It is easy to see that $f$ is an injection but not a bijection, and hence $|A|>|B|$. Consequently $F_T(r)-F_{T'}(r) > 0$, a contradiction.

Similar arguments apply to the case of $r_2$ being an internal vertex. We skip the details.
\end{proof}

For distance problems we also state the following which is a direct consequence of Proposition~\ref{prop:rdis}.

\begin{prop}\label{prop:rdis'}
For a rooted binary tree $T$ with root $r$ and $r_1$, $r_2$ as the children of $r$, with given numbers of descendants of $r_1$ and $r_2$ respectively, $\sigma_{T}(r)$ is maximized when each subtree $T_{r_i}$($i=1,2$), induced by $r_i$ and its descendants (rooted at $r_i$), is a rooted binary caterpillar.
\end{prop}

\begin{proof}
With the above notations direct computation yields
\begin{eqnarray*}
  \sigma_T(r)&=&\sum_{w\in V(T)}d_T(w,r)\\
   &=& \sum_{w\in V(T_{r_1})}(d_T(w,r_1)+1)+\sum_{w\in V(T_{r_2})}(d_T(w,r_2)+1) \\
   &=&\sigma_{T_{r_1}}(r_1)+|T_{r_1}|+\sigma_{T_{r_2}}(r_2)+|T_{r_1}|,
\end{eqnarray*}
maximized if and only if both $\sigma_{T_{r_1}}(r_1)$ and $\sigma_{T_{r_2}}(r_2)$ are maximized. Thus both $T_{r_1}$ and $T_{r_2}$ are rooted binary caterpillars by Proposition~\ref{prop:rdis}.
\end{proof}

\section{Minimum $\sigma_T(w)/\sigma_T(v)$ for $w \in L(T)$ and $v\in C(T)$}
\label{sec:min}

We first establish some characteristics of the tree $T$, the vertices $w \in L(T)$ and $v\in C(T)$, that achieves the minimum $\sigma_T(w)/\sigma_T(v)$.

\subsection{$d_T(w,v)\geq 3$}

In this case, let $P(w,v)$ denote the path connecting $w$ and $v$. Let $x$ be the unique neighbor of $v$ on $P(w,v)$ and let $T_x$ ($T_v$) denote the component containing $x$ ($v$) in $T - xv$. Since $d_T(w,v) \geq 3$, $w \neq x$. Further define $T_w:= T_x \cup \{xv\}$ (Figure~\ref{fig:tv_tw}).

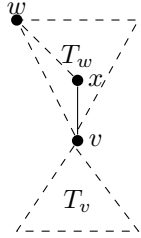
\begin{figure}[htbp]
\centering
    \begin{tikzpicture}[scale=.8]
        \node[fill=black,circle,inner sep=1.5pt] (t1) at (0,0) {};
        \node[fill=black,circle,inner sep=1.5pt] (t2) at (0,1) {};
        \node[fill=black,circle,inner sep=1.5pt] (t4) at (-1,2) {};

        \draw (t1)--(t2);
        \draw [dashed] (t2)--(-1,2);
        \draw [dashed] (t1)--(-1,2)--(1,2)--cycle;
        \draw [dashed] (t1)--(-1,-1.5)--(1,-1.5)--cycle;

        \node at (.3,0) {$v$};
        \node at (.3,1) {$x$};
        \node at (-1,2.2) {$w$};

        \node at (0,-1) {$T_v$};
        \node at (0,1.4) {$T_w$};

        \end{tikzpicture}
\caption{The vertices $v$, $x$, $w$ and components $T_v$, $T_w$.}\label{fig:tv_tw}
\end{figure}

Then, by definition we have
\begin{equation}
\label{eq:0}
\frac{\sigma_T(w)}{\sigma_T(v)}= \frac{\sigma_{T_w}(w) + (n_v-1) d + \sigma_{T_v}(v)}{\sigma_{T_w}(v) + \sigma_{T_v}(v)} = 1 + \frac{\sigma_{T_w}(w) + (n_v-1) d - \sigma_{T_w}(v)}{\sigma_{T_w}(v) + \sigma_{T_v}(v)},
\end{equation}
where $n_v = |V(T_v)|$ and $d = d_T(w,v)$.

Since $\frac{\sigma_T(w)}{\sigma_T(v)} > 1$ in non-trivial trees, to minimize \eqref{eq:0} is equivalent to minimizing the positive expression
\begin{equation}
\label{eq:1}
\frac{\sigma_{T_w}(w) + (n_v-1) d - \sigma_{T_w}(v)}{\sigma_{T_w}(v) + \sigma_{T_v}(v)} .
\end{equation}

Now let $P(w,v) = v u_1 u_2 \ldots u_{d-1} w$. Let $T_i$ denote the component containing $u_i$ in $T-E(P(w,v))$ for $1\leq i \leq d-1$. Since $T$ is a binary tree, $u_i$ ($1\le i\le d-1$) has a neighbor not on $P(w,v)$. Let this neighbor of $u_i$ be $u_i'$. If there exists $1\le j\le d-2$ such that $u_j'\notin L(T)$. Let the two other neighbors of $u_j'$ be $a$ and $b$, consider the tree
$$ T' = T - u_j'a - u_j'b + u''a + u''b, $$
where $u'' \in L(T)$ is a leaf vertex in $T_{d-1}$. Simply,  put $T'$ to be the tree obtained from $T$ by ``moving'' a branch $A$ from the neighbor of $u_j$ to a leaf in $T_{d-1}$ (since $T$ is a binary tree, such a leaf must exist). See Figure~\ref{fig:move}.

\begin{figure}[htbp]
\centering
    \begin{tikzpicture}[scale=.8]
        \node[fill=black,circle,inner sep=1pt] (t1) at (0,0) {};
        \node[fill=black,circle,inner sep=1pt] (t2) at (0,1) {};
        \node[fill=black,circle,inner sep=1pt] (t3) at (0,2) {};
        \node[fill=black,circle,inner sep=1pt] (t4) at (0,3) {};
\node[fill=black,circle,inner sep=1pt] (t21) at (1,1) {};
\node[fill=black,circle,inner sep=1pt] (t22) at (2,2.4) {};

        \draw [dashed] (t3)--(t1);
        \draw (t2)--(1,1);
        \draw (t4)--(t3)--(1,2);
        \draw [dashed] (1,2)--(2,1.6)--(2,2.4)--cycle;
        \draw [dotted] (1,1)--(2,.6)--(2,1.4)--cycle;
        \draw [dotted] (2,2.4)--(3,2)--(3,2.8)--cycle;
        \draw [dashed](t1)--(-1,-1.5)--(1,-1.5)--cycle;

        \node at (0,-1) {$T_v$};
        \node at (0,3.2) {$w$};
\node at (-.3,1) {$u_j$};
\node at (-.45,2) {$u_{d-1}$};
\node at (1,1.4) {$u_j'$};
\node at (2,2.7) {$u''$};
        \draw (2.1,1) edge[out=0,in=180,->, dashed] (2.6,2.05);

        \node[fill=black,circle,inner sep=1pt] (s1) at (0+6,0) {};
        \node[fill=black,circle,inner sep=1pt] (s2) at (0+6,1) {};
        \node[fill=black,circle,inner sep=1pt] (s3) at (0+6,2) {};
        \node[fill=black,circle,inner sep=1pt] (s4) at (0+6,3) {};
\node[fill=black,circle,inner sep=1pt] (s21) at (1+6,1) {};

        \draw [dashed] (s3)--(s1);
        \draw (s2)--(1+6,1);
        \draw (s4)--(s3)--(1+6,2);
        \draw [dashed] (1+6,2)--(3+6,1.2)--(3+6,2.8)--cycle;
        \draw [dashed] (s1)--(-1+6,-1.5)--(1+6,-1.5)--cycle;

\node at (-.3+6,1) {$u_j$};
\node at (-.3+6,0) {$v$};
\node at (1+6,1.4) {$u_j'$};

        \node at (0+6,-1) {$T_v$};
        \node at (1.7, 1) {$A$};
        \node at (2.7,2.4) {$A$};
        \node at (0+6,3.2) {$w$};

        \end{tikzpicture}
\caption{``Moving'' a branch in $T_w$ towards $w$: the operation on $T$ (left) and the resulted $T'$, $T'_w$ (right).}\label{fig:move}
\end{figure}
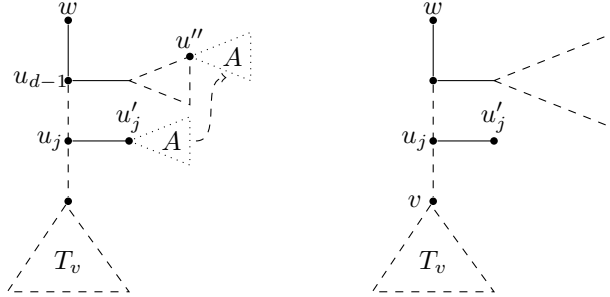

From $T$ to $T'$, we claim the following.

\begin{lemma}\label{lem:1}
With $T$, $T_w$, $T'$, $T'_w$ defined as in Figure~\ref{fig:move}, we have
$$ \sigma_{T'_w}(v) > \sigma_{T_w}(v) $$
and
$$ \sigma_{T'_w}(w) - \sigma_{T'_w}(v) < \sigma_{T_w}(w) - \sigma_{T_w}(v) .$$
\end{lemma}
\begin{proof}
We illustrate our proof with the case $j=1$ and $u'=u'_1$ in Figure~\ref{fig:move}. First by definition we have
\begin{eqnarray}
  \sigma_{T'_w}(v)-\sigma_{T_w}(v) &=& \sum_{z\in A\setminus\{u'\}}(d_T(u'',v)-d_T(u',v)) \nonumber\\
  &=& (|V(A)|-1)(d-1+d_T(u'', u_{d-1})-2)  \nonumber\\
 &=& (|V(A)|-1)(d-3+d_T(u'', u_{d-1})). \label{eqa3}
\end{eqnarray}
Since $|V(A)|>1$,  $d\geq 3$ and $d_T(u'', u_{d-1})\ge 1$, we have
$$
\sigma_{T'_w}(v) > \sigma_{T_w}(v)
$$
as claimed. Similarly,
\begin{eqnarray}
  \sigma_{T'_w}(w)-\sigma_{T_w}(w) &=& \sum_{z\in A\setminus\{u'\}}(d_T(u'',w)-d_T(u',w))   \nonumber\\
   &=& (|V(A)|-1)(1+d_T(u'', u_{d-1})-d)   \nonumber\\
   &=& (|V(A)|-1)(d_T(u'', u_{d-1})-d+1) . \label{eqa4}
\end{eqnarray}
Combining (\ref{eqa3}) and (\ref{eqa4}), after simplification we have
$$
(\sigma_{T'_w}(v)-\sigma_{T_w}(v))-(\sigma_{T'_w}(w)-\sigma_{T_w}(w))=(|V(A)|-1)(2d-4)>0 ,
$$
which is equivalent to
$$ \sigma_{T'_w}(w) - \sigma_{T'_w}(v) < \sigma_{T_w}(w) - \sigma_{T_w}(v) .$$
\end{proof}

Since $T_v$ and $P(w,v)$ stay the same, from Lemma~\ref{lem:1} we see that, from $T$ to $T'$ the numerator of \eqref{eq:1} decreases and the denominator of \eqref{eq:1} increases. Following the same logic one can ``move'' branches from any of the $T_i$ ($1\leq i \leq d-2$) to $T_{d-1}$ and the value of \eqref{eq:1}   decreases. Thus, to minimize \eqref{eq:1} we may assume that $u_i$ has a leaf neighbor for $1\leq i \leq d-2$.

\begin{lemma}
\label{lem:2}
In the tree described above, if $d_T(w,v) \geq 3$, then
$$ \sigma_T(u_1') < \sigma_T(w) . $$
\end{lemma}
\begin{proof}
Let $V^*=\{z: d_T(z, u')\ge d_T(z, w)\}$ be the collection of vertices that are at least as far from $u'$ as from $w$. It is easy to see that $V^*\subset V(T)\setminus \left( V(T_v) \cup \{u_1, u_1'\} \right)$ and for any $z \in V^*$, $d_T(z,u_1')-d_T(z, w)\le d-2$.
Then
\begin{eqnarray*}
  &&\sigma_T(w)-\sigma_T(u_1')\\
  &=& \sum_{z\in T}(d_T(z,w)-d_T(z,u_1'))  \\
   &\ge& \sum_{z\in V(T_v)\cup\{u_1, u_1'\}}(d_T(z,w)-d_T(z,u_1'))+ \sum_{z\in V^*}(d_T(z,w)-d_T(z,u_1'))  \\
   &\ge& (d-2)|V(T_v)|+(d-2)+d-(d-2)|V^*|\\
   &\ge& (d-2)\left( |V(T_v)|+2\right) + 2 - (d-2)(n-|V(T_v)|-2)\\
   &=& (d-2)\left(2|V(T_v)|+4-n\right)+2>0,
\end{eqnarray*}
where the last inequality holds as $|V(T_v)|\ge n-|V(T_v)|$ by Proposition \ref{prop:dis}.
\end{proof}

Since our goal is to minimize \eqref{eq:1} (and naturally picking the leaf vertex with the minimum $\sigma_T(\cdot)$), Lemma~\ref{lem:2} implies that we cannot have $d_T(w,v)\geq 3$.

\subsection{$d_T(w,v)=2$}

Consequently in our extremal structure $w$ must be the neighbor of $x$ as in Figure~\ref{fig:wx1}.

\begin{figure}[htbp]
\centering
    \begin{tikzpicture}[scale=.8]
        \node[fill=black,circle,inner sep=1pt] (t1) at (0,0) {};
        \node[fill=black,circle,inner sep=1pt] (t2) at (0,1) {};
        \node[fill=black,circle,inner sep=1pt] (t3) at (0,2) {};
        \node[fill=black,circle,inner sep=1pt] (t4) at (-1,1) {};

        \draw (t1)--(t2)--(t4)--(t2)--(t3);
        \draw [dashed] (t3)--(-1,3)--(1,3)--cycle;
        \draw [dashed] (t1)--(-1,-1.5)--(1,-1.5)--cycle;

        \node at (.3,0) {$v$};
        \node at (.3,1) {$x$};
        \node at (-1.3,1) {$w$};

        \node at (0,-1) {$T_v$};

        \end{tikzpicture}
\caption{The extremal structure with $d_T(w,v)=2$.}\label{fig:wx1}
\end{figure}
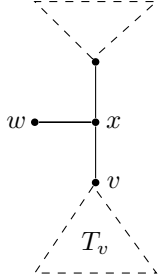

Now $d_{T_w}(w) - d_{T_w}(v) = 0$, $d=2$, \eqref{eq:0} can be rewritten as
\begin{equation}
\label{eq:1'}
\frac{\sigma_T(w)}{\sigma_T(v)}= 1+ \frac{2(n_v - 1)}{\sigma_T(v)} = 1+ \frac{2(n_v - 1)}{\sigma_{T_w}(v) + \sigma_{T_v}(v)} .
\end{equation}

For one thing, with fixed $n_v$ (and hence fixed numbers of vertices in both $T_w$ and $T_v$), \eqref{eq:1'} is obviously minimized when both $\sigma_{T_w}(v)$ and $\sigma_{T_v}(v)$ are maximized. By Propositions~\ref{prop:rdis} and \ref{prop:rdis'}, our extremal structure is a ``3-way caterpillar'' as shown in Figure~\ref{fig:3cat}.

\begin{figure}[htbp]
\centering
    \begin{tikzpicture}[scale=.8]
        \node[fill=black,circle,inner sep=1pt] (t1) at (0,0) {};
        \node[fill=black,circle,inner sep=1pt] (t2) at (0,1) {};
        \node[fill=black,circle,inner sep=1pt] (t3) at (0,2) {};
        \node[fill=black,circle,inner sep=1pt] (t4) at (0,3) {};

        \node[fill=black,circle,inner sep=1pt] (t5) at (1,0) {};
        \node[fill=black,circle,inner sep=1pt] (t6) at (2,0) {};
        \node[fill=black,circle,inner sep=1pt] (t7) at (-1,0) {};
        \node[fill=black,circle,inner sep=1pt] (t8) at (-2,0) {};

        \draw (-3,0)--(t8)--(-2,-1);
        \draw (3,0)--(t6)--(2,-1);
        \draw (-1,-1)--(t7)--(t5)--(1,-1);
        \draw (t1)--(t3)--(-1,2);
        \draw (t2)--(-1,1);
        \draw (-1,3)--(t4)--(0,4);

        \draw [dashed] (t3)--(t4);
        \draw [dashed] (t5)--(t6);
        \draw [dashed] (t7)--(t8);

        \node at (0,-.3) {$v$};
        \node at (.3,1) {$x$};
        \node at (-1.3,1) {$w$};

        \end{tikzpicture}
\caption{The extremal tree $T_1$ with $d_T(w,v)=2$.}\label{fig:3cat}
\end{figure}
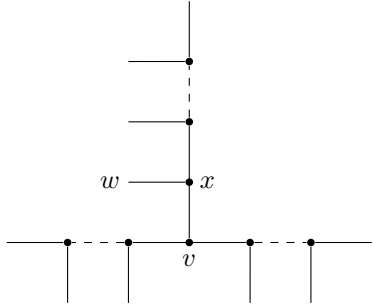

Furthermore, by Proposition~\ref{prop:dis} we have $n_v> n - n_v$ since $v \in C(T)$ and $x \notin C(T)$. Note that $n$ is even (since $T$ is a binary tree), we must have $n_v-2\ge n - n_v$. Hence

\begin{equation}
\label{eq:1''}
\frac{\sigma_T(w)}{\sigma_T(v)}= 1+ \frac{2(n_v - 1)}{\sigma_T(v)} > 1+ \frac{|V(T)| - 1}{\sigma_T(v)}.
\end{equation}

\subsection{$d_T(w,v)=1$}

Now in this case we have, exactly, that
\begin{equation}\label{eq:1'''}
\frac{\sigma_T(w)}{\sigma_T(v)} = 1+ \frac{|V(T)| - 1}{\sigma_T(v)}
\end{equation}
is minimized when $\sigma_T(v)$ is maximized. With given number of vertices on each side of $v$, Proposition~\ref{prop:rdis'} implies that \eqref{eq:1'''} is minimized by when $T$ is a binary caterpillar (Figure~\ref{fig:opt}).

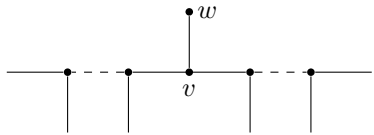
\begin{figure}[htbp]
\centering
    \begin{tikzpicture}[scale=.8]
        \node[fill=black,circle,inner sep=1pt] (t1) at (0,0) {};
        \node[fill=black,circle,inner sep=1pt] (t2) at (0,1) {};
        \node[fill=black,circle,inner sep=1pt] (t5) at (1,0) {};
        \node[fill=black,circle,inner sep=1pt] (t6) at (2,0) {};
        \node[fill=black,circle,inner sep=1pt] (t7) at (-1,0) {};
        \node[fill=black,circle,inner sep=1pt] (t8) at (-2,0) {};

        \draw (-3,0)--(t8)--(-2,-1);
        \draw (3,0)--(t6)--(2,-1);
        \draw (-1,-1)--(t7)--(t5)--(1,-1);
        \draw (t1)--(t2);

        \draw [dashed] (t5)--(t6);
        \draw [dashed] (t7)--(t8);

        \node at (0,-.3) {$v$};
        \node at (.3,1) {$w$};

        \end{tikzpicture}
\caption{The binary caterpillar $T_2$.}\label{fig:opt}
\end{figure}

\subsection{The extremal ratio}

From the above discussion we only need to compare $\frac{\sigma_T(w)}{\sigma_T(v)}$ for $T_1$ (Figure~\ref{fig:3cat}) and $T_2$ (Figure~\ref{fig:opt}). It will be shown that $T_2$ achieves a smaller value for $\frac{\sigma_T(w)}{\sigma_T(v)}$.

\begin{lemma}
\label{lem:com}
With optimized structures $T_1$ and $T_2$ of order $n$, we have
$$ \sigma_{T_2}(v) \geq \sigma_{T_1}(v) $$
with equality if and only if $T_1\cong T_2$.
\end{lemma}

\begin{proof}
First we consider $\sigma_{T_1}(v)$.
Let the other two neighbors of $v$ be $x'$ and $x''$, we denote by $T_x$, $T_{x'}$ and $T_{x''}$ the components containing $x$, $x'$ and $x''$ (respectively) in $T-v$. For $t\in \{x,x',x''\}$, $|V(T_t)|$ is odd since $T_t$ is a rooted binary caterpillar. Furthermore, as $v\in C(T)$ and $x\notin C(T)$, we have that $3\le |V(T_x)|< \frac{n}{2}$ and $|V(T_{x'})|, |V(T_{x''})| \leq \frac{n}{2}$.

By definition we have
\begin{eqnarray*}
  \sigma_{T_1}(v) &=& \sum_{t\in \{x,x',x''\}} \left(1+2\left(2+3+\cdots+\frac{|V(T_t)|+1}{2}\right)\right) \\
  &=&\sum_{t\in \{x,x',x''\}} \frac{(|V(T_t)|+2)^2-5}{4}\\
  &=& \left( \frac14 \sum_{t\in \{x,x',x''\}} (|V(T_t)|+2)^2 \right) -\frac{15}{4}.
\end{eqnarray*}

Next we consider $\sigma_{T_2}(v)$.
Let the other two neighbor of $v$ are $y'$ and $y''$, we use $T_{y'}$ and $T_{y''}$ to denote the components containing $y'$ and $y''$ in $T-v$, respectively. Further assume, without loss of generality, that $|V(T_{y'})|\ge |V(T_{y''})|$. By the fact that $v\in C(T)$ and Proposition~\ref{prop:dis}, we have
\begin{equation*}
    |V(T_{y'})|=\left\{
       \begin{array}{ll}
         \frac{n-2}{2}, & \hbox{$n\equiv 0 \pmod 4$;} \\
         \frac{n}{2}, & \hbox{$n\equiv 2 \pmod 4$}
       \end{array}
     \right.
\end{equation*}
and
\begin{equation*}
    |V(T_{y''})|=\left\{
       \begin{array}{ll}
         \frac{n-2}{2}, & \hbox{$n\equiv 0 \pmod 4$;} \\
         \frac{n-4}{2}, & \hbox{$n\equiv 2 \pmod 4$.}
       \end{array}
     \right.
\end{equation*}
Similar computation yields
\begin{eqnarray*}
  \sigma_{T_2}(v) &=& \frac{(|V(T_{y'})|+2)^2}{4}+\frac{(|V(T_{y''})|+2)^2}{4}-\frac{3}{2}\\
  &=&\left\{
       \begin{array}{ll}
         \frac{n^2+4n-8}{8}, & \hbox{$n\equiv 0 \pmod 4$;} \\
         \frac{n^2+4n-4}{8}, & \hbox{$n\equiv 2 \pmod 4$.}
       \end{array}
     \right.
\end{eqnarray*}
Following simple algebra we have
$$ \sigma_{T_2}(v) \geq \sigma_{T_1}(v) $$
with equality if and only if $T_1\cong T_2$.
\end{proof}

From Lemma~\ref{lem:com} and \eqref{eq:1''}, \eqref{eq:1'''}, we conclude this section with the main result.

\begin{theo}
\label{theo:dis}
Among all binary trees with $n$ (even) vertices, we have
\begin{equation*}
    \min_T \left( \min\limits_{v \in Core(T)\atop w\in L(T)}\frac{\sigma_T(w)}{\sigma_T(v)} \right)= \left\{
       \begin{array}{ll}
         \frac{n^2+12n-16}{n^2+4n-8}, & \hbox{$n\equiv 0 \pmod 4$} \\
         \frac{n^2+12n-12}{n^2+4n-4}, & \hbox{$n\equiv 2 \pmod 4$}
       \end{array}
     \right. ,
\end{equation*}
achieved by the binary caterpillar with $v$ being a centroid vertex  and $w$ being its only leaf neighbor.
\end{theo}

\section{Minimum $F_T(v)/F_T(w)$ for $v \in Core(T)$ and $w\in L(T)$}
\label{sec:minf}

Similar to the last section, we start with examining the characteristics of the extremal structure through different cases.

\subsection{$d_T(v,w)=1$}

Let $T_v$ denote the component containing $v$ in $T - vw$, it is easy to see that
\begin{equation}
\label{eq:2}
\frac{F_T(v)}{F_T(w)} = \frac{2F_{T_v}(v)}{1+F_{T_v}(v)} = 2 - \frac{2}{1+F_{T_v}(v)} < 2.
\end{equation}
Hence, $\frac{F_T(v)}{F_T(w)}$ is minimized when $F_{T_v}(v)$ is minimized (in order to maximize the negative term).

Let the children of $v$ in $T_v$ (as rooted at $v$) be $v_1$ and $v_2$, for $i=1,2$ we use $T_i$ to denote the subtree rooted at $v_i$, induced by $v_i$ and its descendants (Figure~\ref{fig:opt1}).

\begin{figure}[htbp]
\centering
    \begin{tikzpicture}[scale=.8]
        \node[fill=black,circle,inner sep=1pt] (t1) at (-2,0) {};
        \node[fill=black,circle,inner sep=1pt] (t2) at (0,1) {};
        \node[fill=black,circle,inner sep=1pt] (t3) at (0,2) {};
        \node[fill=black,circle,inner sep=1pt] (t4) at (2,0) {};

        \draw (t1)--(t2)--(t3);
        \draw (t4)--(t2);
        \draw [dashed] (t4)--(1,-1.5)--(3,-1.5)--cycle;
        \draw [dashed] (t1)--(-3,-1.5)--(-1,-1.5)--cycle;

        \node at (-.3,1) {$v$};
        \node at (-2.3,0) {$v_1$};
        \node at (-.3,2) {$w$};
        \node at (-2,-1) {$T_1$};
        \node at (2,-1) {$T_2$};
        \node at (2.3,0) {$v_2$};

        \end{tikzpicture}
\caption{Extremal binary tree $T$ with $w$, $v$, $v_1$, $v_2$ and $T_1$, $T_2$.}\label{fig:opt1}
\end{figure}
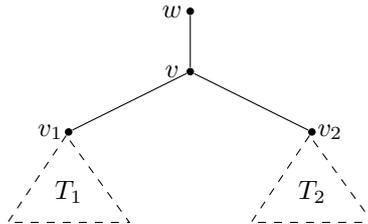

Supposing without loss of generality that $F_{T_1}(v_1) \geq F_{T_2}(v_2)$, by Proposition~\ref{prop:sub} we must have
\begin{equation}
\label{eq:cond1}
2 \left( 1+ F_{T_2}(v_2) \right) \geq F_{T_1}(v_1) \geq F_{T_2}(v_2) .
\end{equation}

Under these conditions, we now examine the lower bound of $F_{T_v}(v)$, and consequently that of $1+F_{T_v}(v) = F_T(w)$.

\begin{lemma}\label{lem:4.1}
Let $T$ be a tree of order $n$, defined above as in Figure~\ref{fig:opt1} under condition \eqref{eq:cond1}, then
\begin{equation}\label{eq:4.1}
    F_{T_v}(v)\ge \left\{
                \begin{array}{ll}
                  9\cdot2^{\frac{n-4}{2}}-3\cdot2^{\frac{n}{4}}+1, & \hbox{$n\equiv0\pmod 4$} \\
                  9\cdot2^{\frac{n-4}{2}}-3\cdot2^{\frac{n-2}{4}}-3\cdot2^{\frac{n-6}{4}}+1, & \hbox{$n\equiv2\pmod 4$}
                \end{array}
              \right.
\end{equation}
with equality if and only if $T_1$ and $T_2$ are rooted binary caterpillars with
$|V(T_1)|=|V(T_2)|=\frac{|n-2|}{2}$ if $n\equiv0\pmod 4$ and $|V(T_1)|=|V(T_2)|+2=\frac{n}{2}$ if $n\equiv2\pmod 4$.
\end{lemma}
\begin{proof}
Let $f_k$ be the number of subtrees containing the root of a rooted binary caterpillar with $k$ internal vertices ($(2k+1)$ vertices). It is easy to see that $f_k=2(1+f_{k-1})$ and $f_1=4$, resulting in $f_k=3\cdot2^k-2$.

Then by Proposition \ref{prop:rsub}, for any rooted binary tree $T'$ with root $r'$, we have
$$F_{T'}(r')\ge f_{\frac{|V(T')|-1}{2}}=3\cdot2^{\frac{|V(T')|-1}{2}}-2$$
with equality if and only if $T'$ is a binary caterpillar rooted at $r'$.

We now consider two different cases depending on the value of $n \mod 4$.

\begin{itemize}
\item[{\bf (A)}] $n\equiv0\pmod 4$:

\begin{itemize}
\item[(A-1)] $|V(T_1)|<|V(T_2)|$.
As $|V(T_1)|+|V(T_2)|=n-2$ and both $|V(T_1)|$ and $|V(T_2)|$ are odd, we have $|V(T_2)|\ge \frac{n+2}{2}$.
Consequently
$$F_{T_1}(v_1)\ge F_{T_2}(v_2)\ge f_{\frac{|V(T_2)|-1}{2}}
\ge f_{\frac{n}{4}}=3\cdot2^{\frac{n}{4}}-2 $$
and
\begin{eqnarray*}
  F_{T_v}(v)&=& (1+F_{T_1}(v_1))(1+F_{T_2}(v_2)) \\
   &\ge & (1+(3\cdot2^{\frac{n}{4}}-2))^2 \\
   &>& 9\cdot2^{\frac{n-4}{2}}-3\cdot2^{\frac{n}{4}}+1.
\end{eqnarray*}

\item[(A-2)] $|V(T_1)|> |V(T_2)|$ and hence $|V(T_1)|\ge \frac{n+2}{2}$.
Then
$$
F_{T_1}(v_1)\ge f_{\frac{|V(T_1)|-1}{2}}\ge f_{\frac{n}{4}}=3\cdot2^{\frac{n}{4}}-2.
$$
Note that by \eqref{eq:cond1}, we have
$$
F_{T_2}(v_2)+1\ge \frac{F_{T_1}(v_1)}{2}.
$$
Hence
\begin{eqnarray*}
  F_{T_v}(v)&=& (1+F_{T_1}(v_1))(1+F_{T_2}(v_2)) \\
   &\ge & \frac{(1+F_{T_1}(v_1))F_{T_1}(v_1)}{2} \\
   &\ge& \frac{(3\cdot2^{\frac{n}{4}}-1)(3\cdot2^{\frac{n}{4}}-2)}{2}\\
   &>& 9\cdot2^{\frac{n-4}{2}}-3\cdot2^{\frac{n}{4}}+1.
\end{eqnarray*}

\item[(A-3)] $|V(T_1)|=|V(T_2)|=\frac{n-2}{2}$. In this case, for $i=1,2$, we have
$$
F_{T_i}(v_i) \geq f_{\frac{|V(T_i)|-1}{2}} = f_{\frac{n-4}{4}}=3\cdot2^{\frac{n-4}{4}}-2
$$
with equality if and only if $T_i$ is a binary caterpillar rooted at $v_i$.
Consequently
$$
F_{T_v}(v)=(1+F_{T_1}(v_1))(1+F_{T_2(v_2)})\ge 9\cdot2^{\frac{n-4}{2}}-3\cdot2^{\frac{n}{4}}+1
$$
with equality if and only if $T_1$ and $T_2$ are rooted binary caterpillars with $|V(T_1)|=|V(T_2)|=\frac{n-2}{2}$.

\end{itemize}

\item[{\bf (B)}] $n\equiv2\pmod 4$: note that $|V(T_1)|\neq |V(T_2)|$ as both $|V(T_1)|$ and $|V(T_2)|$ are odd and $|V(T_1)|+|V(T_2)|=n-2$ is divisible by 4.

\begin{itemize}

\item[(B-1)] $|V(T_1)|\le |V(T_2)|-2$ and hence $|V(T_2)|\ge \frac{n}{2}$.
Then
$$
F_{T_1}(v_1)\ge F_{T_2}(v_2)\ge f_{\frac{|V(T_2)|-1}{2}}
\ge f_{\frac{n-2}{4}}=3\cdot2^{\frac{n-2}{4}}-2
$$
and
\begin{eqnarray*}
  F_{T_v}(v)&=& (1+F_{T_1}(v_1))(1+F_{T_2}(v_2)) \\
   &\ge & (1+(3\cdot2^{\frac{n-2}{4}}-2))^2 \\
   &>& 9\cdot2^{\frac{n-4}{2}}-3\cdot2^{\frac{n-2}{4}}-3\cdot2^{\frac{n-6}{4}}+1.
\end{eqnarray*}

\item[(B-2)] $|V(T_2)|<|V(T_1)|-2$ and hence $|V(T_1)|\ge \frac{n+4}{2}$.
Then
$$F_{T_1}(v_1)\ge f_{\frac{|V(T_1)|-1}{2}}\ge f_{\frac{n+2}{4}} =3\cdot2^{\frac{n+2}{4}}-2.$$
By \eqref{eq:cond1} we have
$$
F_{T_2}(v_2)+1\ge \frac{F_{T_1}(v_1)}{2}.
$$
Consequently
\begin{eqnarray*}
  F_{T_v}(v)&=& (1+F_{T_1}(v_1))(1+F_{T_2}(v_2)) \\
   &\ge & \frac{(1+F_{T_1}(v_1))F_{T_1}(v_1)}{2} \\
   &\ge& \frac{(3\cdot2^{\frac{n+2}{4}}-1)(3\cdot2^{\frac{n+2}{4}}-2)}{2}\\
   &>& 9\cdot2^{\frac{n-4}{2}}-3\cdot2^{\frac{n-2}{4}}-3\cdot2^{\frac{n-6}{4}}+1.
\end{eqnarray*}

\item[(B-3)] $|V(T_2)|= |V(T_1)|-2$ and thus $|V(T_1)|= \frac{n}{2}$, $|V(T_2)|= \frac{n-4}{2}$.
In this subcase we have
\begin{eqnarray*}
  F_{T_v}(v)&=& (1+F_{T_1}(v_1))(1+F_{T_2(v_2)}) \\
      &\ge & (1+f_{\frac{n-2}{4}})(1+f_{\frac{n-6}{4}})\\
   &=& (3\cdot2^{\frac{n-2}{4}}-1)(3\cdot2^{\frac{n-6}{4}}-1)\\
   &=& 9\cdot2^{\frac{n-4}{2}}-3\cdot2^{\frac{n-2}{4}}-3\cdot2^{\frac{n-6}{4}}+1.
\end{eqnarray*}
with equality if and only if $T_1$ and $T_2$ are rooted binary caterpillars with $|V(T_1)|= \frac{n}{2}$ and $|V(T_1)|= \frac{n-4}{2}$.

\end{itemize}
\end{itemize}
\end{proof}

\subsection{$d_T(v,w)=2$}

In this case we let $x \notin Core(T)$ be the common neighbor of $w$ and $v$, and denote by $T_x$ (resp. $T_v$) the component containing $x$ (resp. $v$) in $T-wx - xv$ (Figure~\ref{fig:tv_tx}).

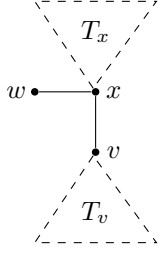
\begin{figure}[htbp]
\centering
    \begin{tikzpicture}[scale=.8]
        \node[fill=black,circle,inner sep=1pt] (t1) at (0,0) {};
        \node[fill=black,circle,inner sep=1pt] (t2) at (0,1) {};
        \node[fill=black,circle,inner sep=1pt] (t4) at (-1,1) {};

        \draw (t1)--(t2)--(t4);
        \draw [dashed] (t2)--(-1,2.5)--(1,2.5)--cycle;
        \draw [dashed] (t1)--(-1,-1.5)--(1,-1.5)--cycle;

        \node at (.3,0) {$v$};
        \node at (.3,1) {$x$};
        \node at (-1.3,1) {$w$};

        \node at (0,-1) {$T_v$};
        \node at (0,2) {$T_x$};

        \end{tikzpicture}
\caption{Binary tree with $T_x$ and $T_v$}\label{fig:tv_tx}
\end{figure}

Now we have
\begin{equation}
\label{eq:2'}
\frac{F_T(v)}{F_T(w)} =  \frac{(2F_{T_x}(x)+1)F_{T_v}(v)}{(F_{T_v}(v) + 1)F_{T_x}(x) + 1}
= 2 - \frac{ (2 F_{T_x}(x) - F_{T_v}(v)) + 2 }{(F_{T_v}(v) + 1)F_{T_x}(x) + 1} .
\end{equation}

Note that by Proposition~\ref{prop:sub}, since $v$ is in the subtree core but $x$ is not, we have $F_{T_v}(v) > 2 F_{T_x}(x) $. Consequently the numerator of \eqref{eq:2'} is
$$ (2 F_{T_x}(x) - F_{T_v}(v))+2 \leq -1 + 2 =1 , $$
implying that \eqref{eq:2'} is strictly less than 2 only when
\begin{equation}
\label{eq:cond2}
F_{T_v}(v) = 2 F_{T_x}(x) + 1,
\end{equation}
in which case
\begin{equation}
\label{eq:2''}
\frac{F_T(v)}{F_T(w)} = 2 - \frac{ (2 F_{T_x}(x) - F_{T_v}(v)) + 2 }{(F_{T_v}(v) + 1)F_{T_x}(x) + 1} = 2 - \frac{ 1 }{(F_{T_v}(v) + 1)F_{T_x}(x) + 1}.
\end{equation}
This is minimized when $(F_{T_v}(v) + 1)F_{T_x}(x) + 1 = F_{T}(w)$ is minimized.  Denote by $y$ the third neighbor of $x$, and $T_y$ the component containing $y$ in $T-xy$ (Figure~\ref{fig:opt2}), we want to minimize $F_{T}(w)$ under condition \eqref{eq:cond2}.

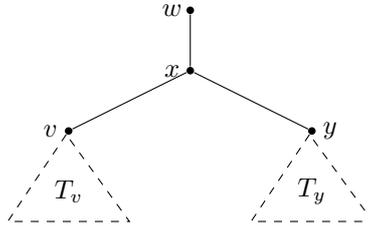
\begin{figure}[htbp]
\centering
    \begin{tikzpicture}[scale=.8]
        \node[fill=black,circle,inner sep=1pt] (t1) at (-2,0) {};
        \node[fill=black,circle,inner sep=1pt] (t2) at (0,1) {};
        \node[fill=black,circle,inner sep=1pt] (t3) at (0,2) {};
        \node[fill=black,circle,inner sep=1pt] (t4) at (2,0) {};

        \draw (t1)--(t2)--(t3);
        \draw (t4)--(t2);
        \draw [dashed] (t4)--(1,-1.5)--(3,-1.5)--cycle;
        \draw [dashed] (t1)--(-3,-1.5)--(-1,-1.5)--cycle;

        \node at (-.3,1) {$x$};
        \node at (-2.3,0) {$v$};
        \node at (-.3,2) {$w$};
        \node at (-2,-1) {$T_v$};
        \node at (2,-1) {$T_y$};
        \node at (2.3,0) {$y$};

        \end{tikzpicture}
\caption{Extremal binary tree $T$ with $w$, $x$, $v$, $y$ and $T_1$, $T_y$.}\label{fig:opt2}
\end{figure}

We now consider the lower bound of $F_T(w)$ for such a tree.

\begin{lemma}\label{lem:4.2}
Let  $T$ be defined as in Figure~\ref{fig:opt2} under condition \eqref{eq:cond2},
we have
\begin{equation}\label{eq:4.2}
    F_{T}(w)\ge \left\{
                \begin{array}{ll}
                  9\cdot2^{\frac{n-2}{2}}-3\cdot2^{\frac{n}{4}}+1, & \hbox{$n\equiv0\pmod 4;$} \\
                  9\cdot2^{\frac{n-4}{2}}-3\cdot2^{\frac{n-2}{4}}+1, & \hbox{$n\equiv2\pmod 4$.}
                \end{array}
              \right.
\end{equation}
\end{lemma}

\begin{proof}
First note that with our notations condition \eqref{eq:cond2} is equivalent to
$$ F_{T_v}(v)=2F_{T_y}(y)+3. $$
Hence on the one hand we have
\begin{eqnarray*}
  F_T(w) &=& 1+(F_{T_v}(v) + 1)F_{T_x}(x)=1+(F_{T_v}(v) + 1)(F_{T_y}(y) + 1) \\
   &=& 1+2(F_{T_y}(y)+2)(F_{T_y}(y)+1).
\end{eqnarray*}
On the other hand we have
$$  F_T(w) =1+(F_{T_v}(v) + 1)(F_{T_y}(y) + 1) = 1 + \frac12 (F_{T_v}(v) + 1)(F_{T_v}(v) - 1) = \frac{(F_{T_v}(v))^2+1}{2} . $$
Similar to before, we now consider two cases depending on $n \mod 4$.

\begin{itemize}
\item[{\bf (A)}] $n \equiv 0 \pmod 4$:
\begin{itemize}
\item[(A-1)] $|V(T_v)|\le |V(T_y)|$. As $|V(T_v)|+|V(T_y)|=n-2$ and both $|V(T_v)|$ and $|V(T_y)|$ are odd, we have that $|V(T_y)|\ge \frac{n-2}{2}$ and $$ F_{T_y}(y)\ge f_{\frac{n-4}{4}}=3\cdot2^{\frac{n-4}{4}}-2 . $$
Consequently
\begin{eqnarray*}
  F_T(w)&=& 1+2(F_{T_y}(y)+2)(F_{T_y}(y)+1) \\
&\ge& 1+2(f_{\frac{n-4}{4}}+2)(f_{\frac{n-4}{4}}+1)\\
   &= &9\cdot2^{\frac{n-2}{2}}-3\cdot2^{\frac{n}{4}}+1.
   \end{eqnarray*}

\item[(A-2)] $|V(T_v)|> |V(T_y)|$ and thus $|V(T_v)|\ge \frac{n+2}{2}$.
Since $F_{T_v}(v)=2F_{T_y}(y)+3$ is odd, we have
$$F_{T_v}(v)\ge f_{\frac{|V(T_v)|-1}{2}}+1= 3\cdot2^{\frac{n}{4}}-1 . $$
Hence
\begin{eqnarray*}
  F_{T}(w)&=& \frac{F_{T_v}(v)^2+1}{2} \\
   &\ge & \frac{(3\cdot2^{\frac{n}{4}}-1)^2+1}{2}\\
   &=& 9\cdot2^{\frac{n-2}{2}}-3\cdot2^{\frac{n}{4}}+1.
\end{eqnarray*}

\end{itemize}

\item[{\bf (B)}] $n \equiv 2 \pmod 4$: note that we have  $|V(T_v)|\neq |V(T_y)|$ as both $|V(T_v)|$ and $|V(T_y)|$ are odd and $|V(T_v)|+|V(T_y)|=n-2$ is divisible by 4.

\begin{itemize}
\item[(B-1)] $|V(T_v)|\le |V(T_y)|+2$ and thus $|V(T_y)|\ge \frac{n-4}{2}$.
Then
\begin{eqnarray*}
 F_{T}(w)&=& 1+2(F_{T_y}(y)+2)(F_{T_y}(y)+1) \\
   &\ge& 1+2(f_{\frac{n-6}{4}}+2)(f_{\frac{n-6}{4}}+1)\\
   &= &9\cdot2^{\frac{n-4}{2}}-3\cdot2^{\frac{n-2}{4}}+1.
\end{eqnarray*}

\item[(B-2)] $|V(T_y)|<|V(T_v)|-2$ and thus $|V(T_v)|\ge \frac{n+4}{2}$.
Note that $F_{T_v}(v)=2F_{T_y}(y)+3$ is odd and
$$F_{T_v}(v)\ge f_{\frac{|V(T_v)|-1}{2}}+1= 3\cdot2^{\frac{n+2}{4}}-1 . $$
Hence
\begin{eqnarray*}
 F_{T}(w)&=& \frac{F_{T_v}(v)^2+1}{2} \\
   &\ge & \frac{(3\cdot2^{\frac{n+2}{4}}-1)^2+1}{2}\\
   &=& 9\cdot2^{\frac{n}{2}}-3\cdot2^{\frac{n+2}{4}}+1.
\end{eqnarray*}

\end{itemize}
\end{itemize}
\end{proof}

\subsection{$d_T(v,w) \geq 3$}

In this case, let $x \notin Core(T)$ be the unique neighbor of $w$ and let $v' \notin Core(T)$ (resp. $v''$) denote the neighbor of $x$ (resp. $v$) on the path $P(v,x)$, further let $T_{v'}$, $T_{v''}$, $T_x$, $T_v$ be the component containing $v'$, $v''$, $x$, $v$ in $T- xv'$, $T-vv''$, $T-E(P(v',w))$, $T- E(P(v,w))$ respectively (Figure~\ref{fig:g2}).

\begin{figure}[htbp]
\centering
    \begin{tikzpicture}[scale=.8]
        \node[fill=black,circle,inner sep=1pt] (t1) at (0,0) {};
        \node[fill=black,circle,inner sep=1pt] (t2) at (0,1) {};
        \node[fill=black,circle,inner sep=1pt] (t3) at (-.5,2) {};
        \node[fill=black,circle,inner sep=1pt] (t4) at (-1,3) {};

        \draw [dashed] (t1)--(t2)--(t3);
\draw (t2)--(t3)--(t4);
        \draw [dashed] (t2)--(t3);
        \draw [dashed] (t2)--(3,-1.5)--(-2,-1.5)--(t2);
        \draw [dashed] (t1)--(-1,-1.5)--(1,-1.5)--cycle;
\draw [dashed] (t3)--(-.5,3)--(1,3)--(t3);

        \node at (.3,0) {$v$};
        \node at (.3,1.2) {$v'$};
        \node at (-.8,2) {$x$};
        \node at (-1.3,3) {$w$};

        \node at (0,-1) {$T_v$};
\node at (.05,2.7) {$T_x$};
        \node at (1.5,-1) {$T_{v'}$};

        \end{tikzpicture}
\caption{Binary tree with $v,v',x,w$, $T_v$, $T_{v'}$, and $T_x$.}\label{fig:g2}
\end{figure}
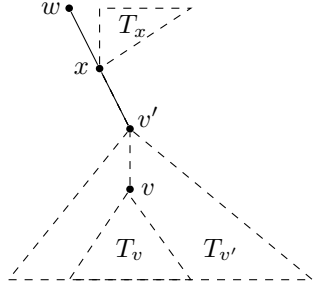

Now we may repeat the argument in the previous section with $w$ and $v'$ and have
$$ \frac{F_T(v')}{F_T(w)} = 2 - \frac{ (2 F_{T_x}(x) - F_{T_{v'}}(v')) + 2 }{(F_{T_{v'}}(v') + 1)F_{T_x}(x) + 1} . $$

It is easy to see that $F_{T_{v'}}(v') >  F_{T_v}(v) $ and $F_{T_{v''}}(v'') \geq 2 F_{T_x}(x) + 1 $. On the other hand, by Proposition~\ref{prop:sub}, $F_{T_v}(v) \geq F_{T_{v''}}(v'')$. Consequently
$$ F_{T_{v'}}(v') >  F_{T_v}(v) \geq F_{T_{v''}}(v'') \geq  2 F_{T_x}(x) + 1  $$
and
$$ (2 F_{T_x}(x) - F_{T_{v'}}(v')) + 2 \leq 0 .$$
Thus
$$ \frac{F_T(v)}{F_T(w)} > \frac{F_T(v')}{F_T(w)} = 2 - \frac{ (2 F_{T_x}(x) - F_{T_{v'}}(v')) + 2 }{(F_{T_{v'}}(v') + 1)F_{T_x}(x) + 1} \geq 2 . $$

\subsection{The extremal ratio}

With the above discussion, to minimize $\frac{F_T(v)}{F_T(w)}$ we only need to compare:
\begin{itemize}
\item[(I)] the minimum value of \eqref{eq:2} for $T$ as in Figure~\ref{fig:opt1} under condition \eqref{eq:cond1}; and
\item[(II)] the minimum value of \eqref{eq:2''} for $T$ as in Figure~\ref{fig:opt2} under condition \eqref{eq:cond2}.
\end{itemize}

From Lemmas~\ref{lem:4.1} and \ref{lem:4.2}, simple algebra shows that case (I) above yields the minimum $\frac{F_T(v)}{F_T(w)}$. This minimum value is
\begin{align*}
 2 - \frac{2}{1+\left( 9\cdot2^{\frac{n-4}{2}}-3\cdot2^{\frac{n}{4}}+1\right) } & = \frac{ 9\cdot2^{\frac{n-2}{2}}-3\cdot2^{\frac{n+4}{4}}+2 }{9\cdot2^{\frac{n-4}{2}}-3\cdot2^{\frac{n}{4}}+2} \\
& =  \frac{ 9\cdot2^{\frac{n-4}{2}}-3\cdot2^{\frac{n}{4}}+1 }{9\cdot2^{\frac{n-6}{2}}-3\cdot2^{\frac{n-4}{4}}+1}
\end{align*}
when $n \equiv 0 \pmod 4$ and
\begin{align*}
 2 - \frac{2}{1+\left( 9\cdot2^{\frac{n-4}{2}}-3\cdot2^{\frac{n-2}{4}}-3\cdot2^{\frac{n-6}{4}}+1 \right) } &
= \frac{ 9\cdot2^{\frac{n-2}{2}}-3\cdot2^{\frac{n+2}{4}}-3\cdot2^{\frac{n-2}{4}} +2 }{9\cdot2^{\frac{n-4}{2}}-3\cdot2^{\frac{n-2}{4}}-3\cdot2^{\frac{n-6}{4}}+2} \\
& =  \frac{ 9\cdot2^{\frac{n-4}{2}}-3\cdot2^{\frac{n-2}{4}}-3\cdot2^{\frac{n-6}{4}}+1 }{9\cdot2^{\frac{n-6}{2}}-3\cdot2^{\frac{n-6}{4}}-3\cdot2^{\frac{n-10}{4}}+1}
\end{align*}
when $n \equiv 2 \pmod 4$.

\begin{theo}
\label{theo:sub}
Among all binary trees $T$ with $n$ (even) vertices, we have
$$ \min_T \left( \min\limits_{v \in Core(T)\atop w\in L(T)} \frac{F_T(v)}{F_T(w)} \right) = \left\{
                \begin{array}{ll}
                 \frac{ 9\cdot2^{\frac{n-4}{2}}-3\cdot2^{\frac{n}{4}}+1 }{9\cdot2^{\frac{n-6}{2}}-3\cdot2^{\frac{n-4}{4}}+1}, & \hbox{$n\equiv0\pmod 4$} \\[5pt]
                  \frac{ 9\cdot2^{\frac{n-4}{2}}-3\cdot2^{\frac{n-2}{4}}-3\cdot2^{\frac{n-6}{4}}+1 }{9\cdot2^{\frac{n-6}{2}}-3\cdot2^{\frac{n-6}{4}}-3\cdot2^{\frac{n-10}{4}}+1} , & \hbox{$n\equiv2\pmod 4$}
                \end{array}
              \right.  , $$
achieved by the binary caterpillar with $v$ being a subtree core vertex (located in the middle) and $w$ being its only leaf neighbor.
\end{theo}

\section{Concluding remarks}
\label{sec:con}

In this paper we studied the extremal ratios $\min \sigma_T(w)/\sigma_T(v)$ (for $w \in L(T)$ and $v\in C(T)$) and $\min F_T(v)/F_T(w)$ (for $v \in Core(T)$ and $w \in L(T)$) for binary trees. This is a first step in examining the correlation between the distance and subtree problems in binary trees through extremal ratios. Indeed we found that both extremal ratios are achieved by binary caterpillars with the leaf $w$ adjacent to the middle vertex $v$, located in the middle of the backbone of the caterpillar.

In fact, our findings seem to suggest that the distance and subtree problems are even better correlated in binary trees than general trees, as the two extremal structure for
$\min \sigma_T(w)/\sigma_T(v)$ and $\min F_T(v)/F_T(w)$ are not quite the same \cite{barefoot, ratio1}. Although both are formed by adding a pendant edge to a path, the locations of the pendant edge are very different. See Table~\ref{tab:1} for a quick comparison.

\begin{table}[htbp]
\caption{Comparison between extremal ratio problems in general trees and binary trees.}\label{tab:1}
\begin{center}
\begin{tabular}
{|c|c|c|}
\hline
 & $\min \sigma_T(w)/\sigma_T(v)$ & $\min F_T(v)/F_T(w)$ \\
\hline
$\hbox{Extremal structures} \atop \hbox{in general trees}$ & $\hbox{path with a leaf added} \atop \hbox{at the middle point}$ & $\hbox{path with a leaf added} \atop
\hbox{at position $\sim \frac{2n}{3}$}$ \\
\hline
$\hbox{Extremal structures} \atop \hbox{in binary trees}$ & $\hbox{binary caterpillar with} \atop \hbox{$v,w$ in the middle}$ & $\hbox{binary caterpillar with} \atop \hbox{$v,w$ in the middle}$ \\
\hline
\end{tabular}
\end{center}
\end{table}

It is worth pointing out that our arguments can be directly generalized to analyze the extremal ratios $\min \sigma_T(w)/\sigma_T(v)$ (for $w \in L(T)$ and $v\in C(T)$) and $\min F_T(v)/F_T(w)$ (for $v \in Core(T)$ and $w \in L(T)$) for $k$-ary trees in general. We skip the technical details.

The natural next step is to consider $\max \sigma_T(w)/\sigma_T(v)$ (for $w \in L(T)$ and $v\in C(T)$) and $\max F_T(v)/F_T(w)$ (for $v \in Core(T)$ and $w \in L(T)$) for binary trees. For general trees, both extremal ratios are achieve by the so-called comet (a tree formed by identifying the end of a path with the center of a star). It seems reasonable to conjecture that the corresponding extremal structures in binary trees are formed by identifying an end of a binary caterpillar (the binary version of a path) and the root of a ``rgood'' binary tree (the binary version of a star). See, for instance, \cite{subtrees, largest} for details on these definitions. However, it appears to be difficult to prove such a statement or provide a counter example.

It is, of course, also interesting to investigate extremal ratios involving the global functions for binary trees. We intend to do exactly that in the near future.


\begin{thebibliography}{99}



  \bibitem{barefoot}  C.A.~Barefoot, R.C.~Entringer, L.A.~Sz\'ekely,
 Extremal values for ratios of distances in trees,
{\em Discrete Appl. Math.} {\bf 80} (1997) 37--56.

\bibitem{cela} E.~Cela, N.S.~Schmuck, S.~Wimer, G.J.~Woeginger, The Wiener maximum quadratic assignment problem, {\em Discrete Optim.} {\bf 8} (2011) 411--416.

\bibitem{gutman}
A.A.~Dobrynin, R.C.~Entringer, I.~Gutman,
Wiener index of trees: Theory and applications,
 {\em Acta Appl. Math.} {\bf 66} (3) (2001) 211--249.

\bibitem{ejs} R.C.~Entringer,  D.E.~Jackson, D.A.~Snyder,
Distance in graphs, {\em Czechoslovak Math. J.} {\bf 26} (101) (1976)
283--296.

\bibitem{rauten}  M.~Fischermann, A.~Hoffmann, D.~Rautenbach, L.~Sz\'ekely
, L. Volkmann,
Wiener index versus maximum degree in trees,  {\em Discrete Appl. Math.}
{\bf 122} (1-3) (2002) 127--137.




\bibitem{kirk} R.~Kirk, H.~Wang, Largest number of subtrees of trees with a given maximum
              degree, {\em SIAM J. Discrete Math.} \textbf{22}(3) (2008) 985--995.

\bibitem{sorder}
F.~Jelen, E.~Triesch,  Superdominance order and distance of
trees with bounded maximum degree, {\em Discrete Appl. Math.}
{\bf 125} (2--3) (2003)   225--233.


\bibitem{jordan} C.~Jordan,  Sur les assemblages de lignes, {\em J. Reine
Angew. Math.} {\bf 70} (1869) 185--190.

\bibitem{dan} B.~Knudsen,   Optimal multiple parsimony alignment with
affine gap cost using a phylogenetic tree,  {\em Lecture Notes in Bioinformatics} {\bf 2812},
 Springer Verlag, 2003,  433--446.

 \bibitem{shuchao}
S. Li, S. Wang, Further analysis on the total number of subtrees of trees, Electron. J. Combin. {\bf 19}(4) (2012), \#P48, 14 pp.


 \bibitem{lovasz} L. Lov\'asz,  {\em Combinatorial Problems and Exercises},
2nd ed., North--Holland Publishing Co., Amsterdam, 1993.

\bibitem{nina} N.~Schmuck, S.~Wagner, H.~Wang, Greedy trees, caterpillars, and Wiener-type graph invariants, {\em MATCH Commun.Math.Comput.Chem.} 68(1) (2012) 273--292.

\bibitem{subtrees} L.A.~Sz\'ekely, H.~Wang, On subtrees of trees,
 {\em Adv. Appl. Math.} {\bf 34} (2005) 138--155.

\bibitem{congnum} L. A. Sz\'ekely, H.~Wang, Binary trees with the
 largest number of subtrees with at least one leaf,
 {\em Congr. Numer.} {\bf 177}  (2005), 147--169.

\bibitem{largest}
L.A.~Sz\'ekely, H.~Wang, Binary trees with the largest number of
subtrees,   {\em Discrete Appl. Math.}  {\bf 155} (3) 2006  374--385.

\bibitem{ratio1}   L.A.~Sz\'ekely, H.~Wang,  Extremal values of ratios: distance problems vs. subtree problems in trees,  {\em Electronic J. Combinatorics} \textbf{20} (1), (2013) \#P67,  20pp.

\bibitem{ratio2}   L.A.~ Sz\'ekely, H.~Wang, Extremal values of ratios: distance problems vs. subtree problems in trees II,   {\em Discrete Math.} \textbf{322} (2014) 36--47.

\bibitem{correl} S.~Wagner, Correlation of graph-theoretical indices,
{\em SIAM J. Discrete Mathematics}  {\bf 21} (1) (2007) 33--46.

\bibitem{degreeseq} H.~Wang, The extremal values of the Wiener index of a tree with given degree sequence, {\em Discrete Appl. Math.} {\bf 156} (14) (2008) 2647--2654, Corrigendum,
{\em Discrete Appl. Math.} {\bf  157} (18) (2009) 3754.

\bibitem{wiener} H.~Wiener, Structural determination of paraffin boiling points, {\em J. Am. Chem. Soc.} {\bf  1} (69) (1947) 17--20.

\bibitem{zelinka} B. Zelinka, Medians and peripherans of trees, {\em Arch. Math.} (Brno) {\bf 4} (1968) 87--95.

\bibitem{zhang2008} X.-D.~Zhang, Q.-Y.~Xiang, L.-Q.~Xu, R.-Y.~Pan, The Wiener index of trees with given degree sequences, {\em MATCH Commun. Math. Comput. Chem.}, {\bf 60} (2008) 623--644.

\bibitem{zhang2010} X.-D.~Zhang,  Y.~Liu and M.-X.~Han, Maximum Wiener Index of Trees with
Given Degree Sequence, {\em  MATCH Commun. Math. Comput. Chem.},
{\bf 64} (2010) 661--682.

\bibitem{gray} X.-M. Zhang,
X.-D. Zhang,
D. Gray,
H. Wang, The number of subtrees of trees with given degree sequence,
{\em J. Graph Theory} {\bf 73} (3) (2013) 280--295.

\bibitem{joc} X.-M. Zhang,
X.-D. Zhang,
D. Gray,
H. Wang, Trees with the most subtrees---an algorithmic approach, {\em J. Comb.} \textbf{3} (2) (2012) 207--223.

\end{thebibliography}
\end{document}